\documentclass[reqno]{amsart}

\usepackage[a4paper]{geometry}

\usepackage[T2A]{fontenc}
\usepackage[utf8]{inputenc}
\usepackage{amsthm, amsmath, amsfonts, amssymb, mathtools}
\usepackage{geometry}
\usepackage{titleps}
\usepackage{soulutf8}
\usepackage{fontsize}
\usepackage{bbm}
\usepackage{cite}

\usepackage[mathscr]{eucal}

\usepackage{mathtools}
\usepackage{enumerate}
\usepackage{accents}
\usepackage{comment}
\usepackage{dsfont}

\usepackage{xcolor}
\usepackage[colorlinks=true]{hyperref}
\hypersetup{urlcolor=blue,citecolor=red,linkcolor=blue}
\usepackage[initials]{amsrefs}

\theoremstyle{plain}
\newtheorem{theorem}{Theorem}[section]
\newtheorem{lemma}{Lemma}[section]
\theoremstyle{remark}

\theoremstyle{definition}
\newtheorem*{definition}{Definition}

\newtheorem{corollary}{Corollary}[section]
\newcommand{\R}{\mathbb R}
\newcommand{\N}{\mathbb N}

\newcommand{\vertiii}[1]{{\left\vert\kern-0.25ex\left\vert\kern-0.25ex\left\vert #1
    \right\vert\kern-0.25ex\right\vert\kern-0.25ex\right\vert}}
    \newcommand{\vertiiis}[1]{{\vert\kern-0.25ex\vert\kern-0.25ex\vert #1
    \vert\kern-0.25ex\vert\kern-0.25ex\vert}}
\newcommand{\scal}[1]{{\left\langle\kern-0.25ex\left\langle #1
    \right\rangle\kern-0.25ex\right\rangle}}

\begin{document}
\title[ Plasticity of the unit spheres]{On the plasticity of the unit spheres of $\ell_1$, $\ell_{\infty}$, $c$, and Hilbert spaces}
\author[M.~Levchenko]{Maksym Levchenko}
\address{Department of Mathematics and Informatics, V. N. Karazin Kharkiv National
University, 61022 Kharkiv, Ukraine.}
\email{lemax2233@gmail.com}

\author[O.~Zavarzina]{Olesia Zavarzina}
\address{Department of Mathematics and Informatics, V. N. Karazin Kharkiv National
University, 61022 Kharkiv, Ukraine.}
\email{olesia.zavarzina@yahoo.com}
\begin{abstract} 
This paper demonstrates the expand-contract plasticity of the unit spheres of $\ell_1$, $\ell_{\infty}$, and $c$. Furthermore, it establishes the strong plasticity of the unit spheres of Hilbert spaces.

\end{abstract}
\thanks{ The second author was supported  by the National Research Foundation of Ukraine, project No. 2025.07/0369 ``Qualitative methods of nonlinear analysis of heterogeneous structures''.}
\maketitle

\section{Introduction}
\begin{definition}
A metric space $(M,\rho)$ is said to be \textit{plastic} if every non-expansive bijective map $F:M \rightarrow M$ is an isometry.
\end{definition}
As demonstrated in \cite{plastms}, a totally bounded space is plastic, but the converse statement is not necessarily true, and there are even examples of unbounded plastic spaces.
According to \cite{plastms}, totally bounded metric spaces possess an even stronger property. Namely, it was shown that in such spaces, an increase in the distance between two points under the action of an arbitrary mapping is compensated by a decrease in the distance between some other pair of points. Later, in \cite{str_pl}, this property was called \textit{strong plasticity}.
\begin{definition}
A metric space $(M,\rho)$ is said to be \textit{strongly plastic} if every non-contractive mapping $F:M \rightarrow M$ is an isometric embedding.
\end{definition}
If $X$ is an infinite-dimensional Banach space, then both its unit ball $B_X$ and the unit sphere $S_X$ are not totally bounded. The unit balls of some Banach spaces have been shown to be plastic \cite{casc,l1ball,l1sum,ang,leoc,2new,fakh}. Whether $B_X$ (or $S_X$) is plastic for every Banach space $X$ remains an open question for the time being.  As the title suggests, this paper deals with the plasticity of the unit spheres of some (real) Banach spaces. \\
Section \ref{HS} is devoted to Hilbert spaces. In  \cite{str_pl}, it was shown that the unit ball of the Hilbert space $\ell_2$ is not strongly plastic, although it is plastic (as are the unit balls of all Hilbert spaces \cite{casc}). Here we show that for a Hilbert space $H$ the  strong plasticity of its unit sphere $S_H$ is a simple consequence of the parallelogram law.  In section \ref{l1} we demonstrate that the unit sphere of the sequence space $\ell_1$ is plastic. In the last section, we use the approach from \cite{leoc} to prove the plasticity of the unit spheres of $\ell_\infty$ and $c$. \\
For an element $x$ of a Banach sequence space, we denote its $n$-th coordinate by $x^n$, and by $\mathtt{supp}(x)$ we denote the set of all such $k \in \N$ that $x^k \neq 0$. For every $k \in \N$, we denote by $e_k$ the vector that has $1$ as its $k$-th coordinate and $0$ as its $j$-th coordinate for every $j \neq k$.\\
For a convex subset $A$ of a Banach space, we denote by $\mathtt{ext}(A)$ the set of extreme points of $A$.

\section{Hilbert spaces}\label{HS}
A Banach space $X$ is called \textit{strictly convex} if its unit sphere $S_X$ contains no non-trivial line segments, i.e. $S_X=\mathtt{ext}(B_X)$. Equivalently, $X$ is strictly convex if and only if $||x-y||=2$ implies that $y=-x$ for every $x,y \in S_X$. In particular, all Hilbert spaces are strictly convex.
\begin{lemma}

Let $X$ be a strictly convex space, and let $F:S_X \rightarrow S_X$ be a non-contractive mapping. 
Then $F(-x)=-F(x)$ for every $x\in S_X$. In particular, this equality holds for non-expansive bijections.
\end{lemma}
\begin{proof}
    Let $x\in S_X$. Then $||x-(-x)||=2 \leq ||F(x)-F(-x)||\leq ||F(x)||+||F(-x)||=2$, so $||F(x)-F(-x)||=2=||F(x)||+||-F(-x)||$, and the strict convexity of $X$ implies that $F(x)=-F(-x)$.\\
    Now suppose that $G:S_X \rightarrow S_X$ is a non-expansive bijection. Then $G^{-1}$ is a non-contractive bijection, and since $G^{-1}$ is surjective, there is an element $y\in S_X$ such that $x=G^{-1}(y)$. Thus, $G(-x)=G(-G^{-1}(y))=G(G^{-1}(-y))=-y=-G(x)$.
\end{proof}

\begin{theorem}
    Let $H$ be a Hilbert space, and let $F:S_H\rightarrow S_H$ be a non-expansive bijection. Then $F$ is an isometry.
\end{theorem}
\begin{proof}
    Let $x,y \in S_H$. Then \[||x-y||^2 +||x+y||^2=||F(x)-F(y)||^2+||F(x)+F(y)||^2=4.\]
    The strict inequality $||x-y||^2>||F(x)-F(y)||^2$ thus implies $||x+y||^2 <||F(x)+F(y)||^2$,\\ a contradiction to the supposed non-expansiveness of $F$.
\end{proof}

 The next corollary follows from the proof of the previous lemma and theorem.
\begin{corollary}
  The unit sphere of any Hilbert space is strongly plastic.
\end{corollary}

\section{The space $\ell_1$}\label{l1}

\begin{lemma}\label{l1lem1}
    Let $p,q \in S_{\ell_1}$ be such that for every $x \in S_{\ell_1}$
    \begin{equation*}
        (||x-p||=2)\vee (||x-q||=2).
    \end{equation*}
    Then $p=\pm e_k$ for some $k \in \N$, and $q=-p$.
\end{lemma}

\begin{proof}
    It is enough to show that the following two statements are true for such $p,q$:
    \begin{equation}
        \mathtt{supp}(p)=\mathtt{supp}(q),
    \end{equation}
    and
    \begin{equation}
        |\mathtt{supp}(p)|=1.
    \end{equation}
    In contradiction to (1) suppose WLOG that $\mathtt{supp}(p) \nsubseteq \mathtt{supp}(q)$, i.e. there is a $k \in \N$ such that $p^k \neq 0$, $q^k =0$, and let $x=\frac12(q+\mathtt{sgn}(p^k)e_k)$. Then $||x-p||<2$, $||x-q||<2$.\\
    Now in contradiction to (2) suppose that $|\mathtt{supp}(p)| \geq 2$, and let $k,j \in \mathtt{supp}(p)=\mathtt{supp}(q)$,  $k\neq j$, $x=\frac12(\mathtt{sgn}(p^k)e_k+\mathtt{sgn}(q^j)e_j)$. Then $||x-p||<2$, $||x-q||<2$.
\end{proof}

\begin{lemma}\label{l1lem2}
    Let $F:S_{\ell_1}\rightarrow S_{\ell_1}$ be a non-expansive bijection. Then for every $x \in \mathtt{ext}(B_{\ell_1})$ $F^{-1}(x)\in \mathtt{ext}(B_{\ell_1})$, and $F^{-1}(-x)=-F^{-1}(x)$.
\end{lemma}
\begin{proof}
    Recall that $\mathtt{ext}(B_{\ell_1})=\{\pm e_k\}_{k=1}^\infty$, and let $x \in \mathtt{ext}(B_{\ell_1})$.\\ Then $(||y-x||=2)\vee (||y+x||=2)$ for every $y \in S_{\ell_1}$, and thus $||F^{-1}(y)-F^{-1}(x)||=2$ or $||F^{-1}(y)-F^{-1}(-x)||=2$ for every $y \in S_{\ell_1}$. Since $F^{-1}$ is bijective, the last statement is equivalent to $(||y-F^{-1}(x)||=2)\vee (||y-F^{-1}(-x)||=2)$ holding for all $y \in S_{\ell_1}$. Applying Lemma \ref{l1lem1} completes the proof.
\end{proof}

\begin{lemma}\label{l1lem3}
    Let $F:S_{\ell_1}\rightarrow S_{\ell_1}$ be a non-expansive bijection, $N \in \N$.
    Then 
    \begin{equation}
        F\left(\sum_{k=1}^Na_kF^{-1}(e_k) \right)=\sum_{k=1}^Na_ke_k
    \end{equation}
    for every collection $\{a_k\}_{k=1}^{N}$ of non-negative real scalars such that $\sum_{k=1}^Na_k=1$.
\end{lemma}

\begin{proof}
     The case $N=1$ is trivial. Now assume that the lemma holds for some $N-1$, $N \geq 2$. We will show that the lemma holds for $N$, thus proving it by induction.\\
     For every $n \in N$ denote $F^{-1}(e_n)$ by $g_n$, and fix some $x=\sum_{k=1}^{N}a_kg_k \in S_{\ell_1}$ such that $a_N\neq0$, $a_k \geq 0$ for every $k \leq N$.
     Denote $(1-a_N)^{-1}\sum_{k=1}^{N-1}a_kg_k \in S_{\ell_1}$ by $\tilde{x}$, so that $x=(1-a_N)\tilde{x}+a_Ng_N$. Then 
     \[||x-g_k||=1-a_k+\sum_{j\neq k}a_j=2(1-a_k)\]
     for every $k \leq N$;
     \[||x-\tilde{x}||=a_N+\sum_{k=1}^{N-1}\left|a_k-\frac{a_k}{1-a_N}\right|=a_N+\left(\frac{1}{1-a_N}-1 \right)\sum_{k=1}^{N-1}a_k=a_N+1-1+a_N=2a_N;\]
     Now let $y=F(x)$, so $||y-e_N||\leq 2(1-a_N)$, $||y-F(\tilde{x})||\leq 2a_N$. Then $y^N\geq 0$, and
     \[||y-e_N||=|y^N-1| +\sum_{k=1}^{N-1}|y^k|=1-y^N+1-y^N=2(1-y^N)\leq 2(1-a_N),\]
     so $y^N \geq a_N$. Then
     \[||y-F(\tilde x)||=\sum_{k=1}^{N-1}\left|y^k-\frac{a_k}{1-a_N}\right|+\sum_{k=N}^\infty|y^k|\leq 2a_N \leq 2\sum_{k=N}^\infty|y^k|,\]
     so
     \[\sum_{k=1}^{N-1}\left|y^k-\frac{a_k}{1-a_N}\right|\leq \sum_{k=N}^\infty|y^k|=1-\sum_{k=1}^{N-1}|y^k|,\]
     \[\sum_{k=1}^{N-1} \left( \left|\frac{a_k}{1-a_N}-y^k\right|+|y^k| \right)\leq 1.\]
     
     On the other hand,
     \[\sum_{k=1}^{N-1} \left( \left|\frac{a_k}{1-a_N}-y^k\right|+|y^k| \right)\geq \sum_{k=1}^{N-1}  \left|\frac{a_k}{1-a_N}-y^k+y^k\right| =(1-a_N)^{-1}\sum_{k=1}^{N-1}|a_k|=1,\]
     so 
     \begin{equation}
        \sum_{k=1}^{N-1} \left( \left|\frac{a_k}{1-a_N}-y^k\right|+|y^k| \right)= 1. 
     \end{equation}
     The inequality $y^N\geq a_N$ means that $\sum_{k=1}^{N-1}|y^k|\leq 1-a_N$, together with $(4)$ it implies that $\sum_{k=1}^{N-1}|\frac{a_k}{1-a_N}-y^k|\geq a_N$. Thus
     \[2a_N\geq ||F(\tilde x)-y||\geq y^N+\sum_{k=1}^{N-1}\left|\frac{a_k}{1-a_n}-y^k\right|\geq 2a_N,\]
     so all the inequalities in the chain above are in fact equalities, $y^N=\sum_{k=1}^{N-1}|\frac{a_k}{1-a_N}-y^k|=a_N$, $\sum_{k=1}^{N-1}|y^k|=1-a_N$, and thus $y=a_Ne_N+\sum_{k=1}^{N-1}y^ke_k$.\\
     Now let $A$ be the set of all such $k\leq N-1$ that $a_k \neq 0$. Then  for every $k \in A$
     \[||y-e_k||=1-y^k+\sum_{j\neq k}|y^j|=1-y^k+1-|y^k|=2(1-y^k)\leq 2(1-a_k),\]
     so $y^k\geq a_k$ for all $k \in A$. Thus
     \[1-a_N=\sum_{k=1}^{N-1}a_k=\sum_{k=1}^{N-1}|y^k|\geq\sum_{k\in A}y^k\geq \sum_{k \in A}a_k=1-a_N,\]
     so $\sum_{k\in A}y^k=\sum_{k\in A}a_k=1-a_N$, and $y=\sum_{k=1}^Na_ke_k$.
    
\end{proof}

\begin{lemma}\label{l1lem4}
    Let $F:S_{\ell_1} \rightarrow S_{\ell_1}$ be a non-expansive bijection, $N \in \N$.\\ Then (3) holds for every collection $\{a_k\}_{k=1}^N$ of real scalars such that $\sum_{k=1}^N|a_k|=1$.
\end{lemma}

\begin{proof}
    Fix some $\{a_k\}_{k=1}^N \in \R^N$ such that $\sum_{k=1}^N|a_k|=1$. Then 
    \[F\left(\sum_{k=1}^Na_kF^{-1}(e_k)\right)=F\left(\sum_{k=1}^N|a_k|F^{-1}(\mathtt{sgn}(a_k)e_k)\right).\]
    Let $T:\ell_1\rightarrow \ell_1$ be a linear isometry defined on the canonical basis by \[Te_k=\begin{cases}\mathtt{sgn}(a_k)e_k, \quad k\leq N,\\
    e_k, \quad k>N,    
    \end{cases}\]
    and denote the restriction of $T$ to $S_{\ell_1}$ by $\tilde T$.\\
    Note that $T^{-1}=T$, and let $G=\tilde T F \tilde T$, so that $F=\tilde T G \tilde T$. Then 
    \begin{equation*}
        \begin{split}
            F\left(\sum_{k=1}^N|a_k|F^{-1}(\mathtt{sgn}(a_k)e_k)\right)  & =\tilde T G \tilde T\left( \sum_{k=1}^N|a_k|\tilde T G^{-1} \tilde T(\mathtt{sgn}(a_k)e_k)\right) \\
            & = \tilde T G \tilde T\left( \sum_{k=1}^N|a_k|\tilde T G^{-1} (e_k)\right) 
            = \tilde T G \tilde T^2\left( \sum_{k=1}^N|a_k| G^{-1} (e_k)\right) \\
            & = \tilde T G \left( \sum_{k=1}^N|a_k| G^{-1} (e_k)\right)=\tilde T \left( \sum_{k=1}^N|a_k|e_k \right) \\
            & = \sum_{k=1}^N|a_k|\mathtt{sgn}(a_k)e_k=\sum_{k=1}^Na_ke_k.
        \end{split}
    \end{equation*}
\end{proof}

\begin{corollary}
    Let $F:S_{\ell_1}\rightarrow S_{\ell_1}$ be a non-expansive bijection. Then $\{F^{-1}(e_k)\}_{k\in \N}$ is a Schauder basis in $\ell_1$.
\end{corollary}

For every non-expansive bijection $F:S_{\ell_1}\rightarrow S_{\ell_1}$ 
let $T_F:\ell_1\rightarrow \ell_1$ be the linear isometry defined on the canonical basis by \[T_Fe_k=F^{-1}(e_k), \quad k\in \N,\]
and let $I_F=F \tilde{T}_F$, where $\tilde{T}_F$ is the restriction of $T_F$ to $S_{\ell_1}$. Then $I_F:S_{\ell_1}\rightarrow S_{\ell_1}$ is a\\ non-expansive bijection, and $I_F(x)=x$ for every $x \in \mathtt{ext}(B_{\ell_1})$.

\begin{lemma}\label{l1lem5}
    Let $F:S_{\ell_1}\rightarrow S_{\ell_1}$ be a non-expansive bijection. Then $I_F(x)=x$ for every $x \in S_{\ell_1}$. 
\end{lemma}

\begin{proof}
    Let $x \in S_{\ell_1}$, $y=I_F(x)$. Then \[|x^k|=\max(1-\frac{||x-e_k||}2,1-\frac{||x+e_k||}2)=1-\frac{||x-e_k||}2+1 - \frac{||x+e_k||}2,\]
    so \[1=||x||  =\sum_{p\in \mathtt{ext}(B_{\ell_1})}\left(1-\frac{||x-p||}2\right)  =\sum_{p\in \mathtt{ext}(B_{\ell_1})}\left(1-\frac{||y-p||}2\right)=||y||.\] 
    
    The existence of such $p \in \mathtt{ext}(B_{\ell_1})$ that $||I_F(x)-I_F(p)||=||y-p||<||x-p||$ would imply the existence of such $p' \in \mathtt{ext}(B_{\ell_1})$ that $||y-p'||=||I_F(x)-I_F(p')||>||x-p'||$, contradicting the non-expansiveness of $I_F$. Thus $||x-p||=||y-p||$ for every $p \in \mathtt{ext}(B_{\ell_1})$, so \[x^k=\frac12(||x+e_k||-||x-e_k||)=\frac12(||y+e_k||-||y-e_k||)=y^k\] for every $k \in \N$.
\end{proof}

\begin{theorem}\label{l1main}
   Let $F:S_{\ell_1}\rightarrow S_{\ell_1}$ be a non-expansive bijection. Then $F$ is an isometry.
\end{theorem}
\begin{proof}
    Lemma \ref{l1lem5} states that the composition of $F$ and $\tilde{T}_F$ is the identity map on $S_{\ell_1}$. Since $\tilde{T}_F$ is an isometry, $F$ has to be an isometry as well.
\end{proof}

\section{The spaces $\ell_\infty$ and $c$}\label{linf}

Throughout this section we denote by $X$ both $\ell_\infty$ and its subspace $c$. As we shall see later, basically the same proof works for both of these spaces. By $E$ we denote the subset $\{\pm e_k\}_{k \in \N}$ of $S_X$.

\begin{lemma}\label{linf1}
    Let $p,q \in S_X$ be such that for every $x \in S_X$
    \begin{equation*}
        (||x-p||\leq1)\vee(||x-q||\leq1).
    \end{equation*}
    Then $p \in E$, $q=-p$.
\end{lemma}

\begin{proof}
    Let us demonstrate that $|\mathtt{supp}(p)|=|\mathtt{supp}(q)|=1$. WLOG suppose the contrary, i.e. there are $k,j \in \mathtt{supp}(p)$, $k\neq j$, and let $l \in \mathtt{supp}(q)$, $l\neq k$. Now let $x \in S_X$ be such that $x^k=-\mathtt{sgn}(p^k)$, $x^l=-\mathtt{sgn}(q^l)$. Then $||x-p||\geq 1+|p^k|>1$, $||x-q||\geq 1+|q^l|>1$, and that is a contradiction. The rest of the proof is trivial.
\end{proof}

\begin{lemma}\label{linf2}
    Let $F: S_X\rightarrow S_X$ be a non-expansive bijection, $x \in E$. Then $F(x)\in E$, $F(-x)=-F(x)$.
\end{lemma}

\begin{proof}
    For every $y \in S_X$ either $||y-x||\leq 1$ or $||y+x|| \leq 1$, and the non-expansiveness of $F$ implies that $\mathtt{min}(||F(y)-F(x)||,||F(y)-F(-x)||)\leq 1$ for every $y \in S_X$. F is bijective, so $\mathtt{min}(||y-F(x)||,||y-F(-x)||)\leq 1$ for every $y \in S_X$. Now all that remains is to apply the previous lemma.
\end{proof}

Recall that \[\mathtt{ext}(B_{\ell_\infty})=\{(x^k):x^k \in \{-1,1\}\},\]
and $\mathtt{ext}(B_c)=c \cap \mathtt{ext}(B_{\ell_\infty})$, i.e. $(x^k)\in\mathtt{ext}(B_c)$ if and only if $x^k\in\{-1,1\}$ for every $k\in \N$, and there is an $N \in \N$ such that $x^k=x^N$ for all $k \geq N$.

\begin{lemma}\label{linf3}
    Let $p,q \in S_X$ be such that for every $x \in S_X$
    \begin{equation*}
        (||x-p||=2)\vee(||x-q||=2).
    \end{equation*}
    Then $p \in \mathtt{ext}(B_X)$, $q=-p$.
\end{lemma}

\begin{proof}
    If there is a $k \in \N$ such that $|p^k|<1$, then $||e_k-p||<2$, $||-e_k-p||<2$, so $||e_k-q||=||-e_k-q||=2$, and that is impossible in $S_X$. Thus $|p^k|=|q^k|=1$ for every $k \in \N$, and there is no $j \in \N$ such that $p^j=q^j$, since in the opposite case both distances $||p^je_j-p||$ and $||p^je_j-q||$ would have been equal to $1$.
\end{proof}

\begin{corollary}
Let $F:S_X\rightarrow S_X$ be a non-expansive bijection, $x \in \mathtt{ext}(B_X)$.\\ Then $F^{-1}(x) \in \mathtt{ext}(B_X)$, $F^{-1}(-x)=-F^{-1}(x)$.
\end{corollary}

For every $a \in \mathtt{ext}(B_X)$ let $T_a:S_X \rightarrow S_X$ be the isometry of pointwise multiplication by $a$:
\[T_a(x^k)=(a^kx^k), \quad (x^k) \in S_X.\]
Now for every non-expansive bijection $F:S_X \rightarrow S_X$ let $F_1={ T_{F^{-1}(e)}}F$, where $e=(1,1,1,...)$. The map $F_1:S_X \rightarrow S_X$  is then a non-expansive bijection, $F_1(e)=e$, $F_1(-e)=-e$.\\
For every $k \in \N$ let $h_k=e-2e_k$, and denote the set $\{h_k\}_{k\in \N}$ by $H$.

\begin{lemma}\label{linf4}
    Let $F: S_X\rightarrow S_X$ be a non-expansive bijection, $x \in H$. Then $F_1^{-1}(x)\in H$.
\end{lemma}

\begin{proof}
    Denote by $W$ the set of all vectors $y \in S_X$ such that
    \[|\{k\in \N:y^k \leq 0\}|=|\{k \in \N:||y+e_k|| \leq 1\}|\leq 1.\]
    Then $W \cap \mathtt{ext}(B_X)= \{e\} \cup H$, so it is enough to show that $F_1^{-1}(x) \in W$. Now assume the contrary, i.e. there are $i,j \in \N$ such that $i \neq j$, $||F_1^{-1}(x)+e_i||\leq 1$, $||F_1^{-1}(x)+e_j||\leq 1$. The fact that $-e$ is a fixed point of $F_1$ implies that $F_1(-e_i)=-e_k$, $F_1(-e_j)=-e_l$ for some $k,l \in \N$. Then $||x+e_k||\leq 1$, $||x+e_l||\leq1$, contradicting the assumption that $x \in H$.
\end{proof}

\begin{lemma}\label{linf5}
     Let $F: S_X\rightarrow S_X$ be a non-expansive bijection. Then $F_1$ maps the set $\{e_k\}_{k \in \N}$ bijectively onto itself.
\end{lemma}

\begin{proof}
    Note that there is an injective map $\sigma:\N \rightarrow \N$ such that $F_1(e_k)=e_{\sigma(k)}$ for every $k \in \N$. We will show that $\sigma$ is in fact a bijection. Now fix some $k \in \N$. Then $||e_{\sigma(k)}-h_{\sigma(k)}||=2$, and so $||F_1^{-1}(e_{\sigma(k)})-F_1^{-1}(h_{\sigma(k)})||=||e_k-F_1^{-1}(h_{\sigma(k)})||=2$. Lemma \ref{linf4} states that $F_1^{-1}(h_{\sigma(k)})=h_j$ for some $j\in \N$. Note that $||e_k-h_j||=1$ for all $j \neq k$, so $j=k$. Thus $F_1^{-1}(\{h_{\sigma(k)}\}_{k\in \N})=H$, so $\sigma(\N)=\N$.
\end{proof}

\begin{corollary}
    Every non-expansive bijection $F: S_X\rightarrow S_X$ maps $E$ bijectively onto itself.
\end{corollary}

Lemma \ref{linf5} states that for every non-expansive bijection $F: S_X\rightarrow S_X$ there is a bijection\\ $\sigma_F: \N \rightarrow \N$ such that $F_1(e_k)=e_{\sigma_F(k)}$ for every $k \in \N$. Now for every non-expansive bijection $F: S_X\rightarrow S_X$ define the bijective isometry $P_F:S_X\rightarrow S_X$ by
    \[P_F(x^k)=(x^{{\sigma_F}^{-1}(k)}), \quad (x^k) \in S_X,\]
    and let $I_F=P_FF_1$. Then $I_F$ is a non-expansive bijection, and $I_F(x)=x$ for every $x \in E$.

\begin{lemma}\label{linf6}
    Let $F: S_X\rightarrow S_X$ be a non-expansive bijection, $x \in \mathtt{ext}(B_X)$. Then $x$ is a fixed point of $I_F$.
\end{lemma}

\begin{proof}
    Let $y=I_F^{-1}(x)$. Then $y \in \mathtt{ext}(B_X)$, and the non-contractiveness of $I_F^{-1}$ implies that $y=x$. Indeed, suppose in contradition that there is a $k \in \N$ such that $y^k=-x^k$. Then $||I_F^{-1}(x)-I_F^{-1}(y^ke_k)||=||y-y^ke_k||=1<2=||x-y^ke_k||$. Thus $I_F^{-1}(x)=x=I_F(x)$. 
\end{proof}

\begin{lemma}\label{linf7}
    Let $F: S_X\rightarrow S_X$ be a non-expansive bijection, $x\in S_X$, $y=I_F(x)$. Then for every $k \in \N$ the following inequalities hold:
    \begin{equation}
        |y^k|\leq |x^k|,
    \end{equation}
    and
    \begin{equation}
        \mathtt{sgn}(x^k)  \mathtt{sgn}(y^k) \geq 0.
    \end{equation}
\end{lemma}

\begin{proof}
    Let $k \in \N$ be such that $|x^k|<1$, i.e. $\mathtt{max}(||x-e_k||,||x+e_k||)<2$. The non-expansiveness of $I_F$ then implies that $\mathtt{max}(||y-e_k||,||y+e_k||)<2$, i.e. $|y^k|<1$, and 
    \[|x^k|=||x-e_k||+||x+e_k||-2\geq ||y-e_k||+||y+e_k||-2=|y^k|.\]
    This proves (5) for all $k\in \N$ such that $|x^k|<1$, and (5) is obviously true for all $k\in \N$ such that $|x^k|=1$. Now suppose the contrary to (6), i.e. there is a $k\in \N$ such that $\mathtt{sgn}(x^k)  \mathtt{sgn}(y^k)=-1$, and let $p \in \mathtt{ext}(B_X)$ be such that $p^j=\mathtt{sgn}(x^j)$ for all $j \in \mathtt{supp}(x)$. Then $||y-p||>1\geq||x-p||$, condradicting the fact that $I_F$ is non-expansive.
\end{proof}

For every $n \in \N$ let $M_n\subset S_X$ be defined as the set of all vectors $x \in S_X$ for which there exist disjoint subsets $A_0,A_1,...,A_n$ of $\N$ and real non-negative scalars $r_1,r_2,...,r_n$ such that $|x^i|=r_k$ for all $i \in A_k$, $1\leq k \leq n$, and

\begin{equation}
    1>r_1 > r_2 >...> r_n,
\end{equation}

\begin{equation}
    \bigsqcup_{n\in \N }A_n = \N \setminus A_0,
\end{equation}

\begin{equation}
    |A_0|=|\{k \in \N :|x^k|=1\}|=1.
\end{equation}
Here the only difference between $c$ and $\ell_\infty$ is that for $X=c$ all the sets $A_k$ except one are finite.
Let the union of all such sets $M_n$ be denoted by $M$.

\begin{lemma}\label{lem_main}
    Let $F: S_X\rightarrow S_X$ be a non-expansive bijection. Then the elements of $M$ are fixed points of $I_F$.
\end{lemma}
\begin{proof}
    We are going to prove by induction that the elements of $M_n$ are fixed points of $I_F$ for every natural $n$.
    Let $x \in M_1$, $y = I_F(x)$. Then there is an $r \in (-1,1)$ such that $x^k=r$ for every $k\in \N \setminus \{N\}$ for some $N \in \N$. If $r=0$, then $x=\pm e_N$, and thus $x$ is a fixed point of $I_F$, so we may WLOG assume that $|r|\in (0,1)$. It is implied by (5) that $|y^k|\leq |r|<1$ for every $k \in \N \setminus \{N\}$, so $|y^N|=1=|x^N|$, and (6) means that $y^N=x^N$. Now let $p\in \mathtt{ext}(B_X)$, $p^k=\mathtt{sgn}(x^k)$. The non-expansiveness of $I_F$ and (6) together mean that $||p-x||=1-|r|\geq ||p-y||\geq 1-|y^k|$ for every $k \in \N$, thus $x=y=I_F(x)$.
    Now assume that for some natural $n \geq 2$ all the elements of $M_{n-1}$ are fixed points of $I_F$, and let $x \in M_{n}$, $y=I_F(x)$. Let $A_0,A_1,...A_n\subset \N$ and $r_1,r_2,...,r_n\in [0, 1)$ be such that $|x^i|=r_k$ for all $i \in A_k$, $1 \leq k \leq n$, with conditions (7-9) satisfied. First we are going to demonstrate that $x^i=y^i$ for all $i\in A_k$, $2 \leq k \leq n$. Let $2 \leq k \leq n$, and define $\tilde{x} \in S_X$ by 
     \[\tilde{x}^i=\begin{cases} \mathtt{sgn}(x^i) \frac12  (r_{k-1}+r_k), \quad i \in A_{k-1} \sqcup A_k,\\
    x^i, \quad i \in \N \setminus (A_{k-1} \sqcup A_k).
    \end{cases}\]
    Then $\tilde{x}\in M_{n-1}$, so $\tilde{x}$ is a fixed point of $I_F$, and \[||\tilde{x}-y|| \leq ||\tilde{x}-x||=|\tilde{x}^i|-|x^i|\leq |\tilde{x}^i|-|y^i| \leq ||\tilde{x}-y||\]
    for all $i \in A_k$.\\
    Now all that remains is to show that $x^i=y^i$ for every $i \in A_1$. 
    Suppose in contradiction that $|x^i|>|y^i|$ for some $i \in A_1$, and let $\lambda \in (r_1,1)$ be such that $1-\lambda <|x^i|-|y^i|$. Now let $p \in M_n$ be defined by
    \[p^j=\begin{cases} \lambda \mathtt{sgn}(x^j) , \quad j \in A_1,\\
    x^j, \quad j \in \N \setminus A_1,
    \end{cases}\]
    and let $q=I_F^{-1}(p)$.
    We are going to show that
   \begin{equation}\label{eqns}
       q^j=p^j=x^j=y^j
   \end{equation} 
    for every $j \in \N \setminus A_1$, so that 
    \begin{equation*}
        ||p-q||=\sup_{j\in A_1} |p^j-q^j|\leq 1-\lambda.
    \end{equation*}
    Fix some arbitrary $j \in \N \setminus (A_1 \sqcup A_0)$, and let $\varepsilon \in (0,1-\lambda)$. Now define $\tilde p \in M_{n-1}$ by 
    \[\tilde{p}^l=\begin{cases} \mathtt{sgn}(p^l) , \quad l=j,\\
    (\lambda+\varepsilon) \mathtt {sgn}(x^l), \quad l \in A_0 \sqcup A_1 \sqcup (A_2 \setminus  \{j\}),\\
    p^l + \varepsilon \mathtt{sgn}(p^l), \quad l \in A_k \setminus \{j\}, \ k \geq 3.
    \end{cases}\]
Then 
\[|p^l-\tilde{p} ^l|=\begin{cases} 1-|p^l| , \quad l=j,\\
   1- \lambda-\varepsilon , \quad l \in A_0 ,\\
    \lambda +\varepsilon -|p^l|, \quad l \in A_2 \setminus \{j\}, \\
    \varepsilon, \quad l \in A_1 \sqcup(A_k \setminus \{j\}), \ k \geq 3,
    \end{cases}\]
so $|p^j-\tilde{p}^j|=1-|p^j|>|p^l-\tilde{p} ^l|$ for every $l \neq j$. If $j \in A_k$, $2 \leq k \leq n$, and $l \in A_m$, $0 \leq m \leq k$, $l \neq j$, then 
\begin{equation*}
    |q^l-\tilde p ^l| \leq \begin{cases}
        1 - \lambda - \varepsilon, \quad l \in A_0,\\
        \max (\varepsilon, 1-\lambda-\varepsilon), \quad l \in A_1,\\
        \max(\lambda +\varepsilon -r_2, 1-\lambda-\varepsilon), \quad l \in A_2,\\
        \max(\varepsilon,1-r_m-\varepsilon), \quad l \in A_m, \ 2<m\leq k,
    \end{cases}
\end{equation*}
so
\begin{equation}
    |q^l-\tilde p^l|<1-|p_j|=|p^l-\tilde p^j|=||p-\tilde p||.
\end{equation}
Thus (10) follows from (11) for $j \in A_n$, and for $j \in A_k$, $2 \leq k <n$, it follows from (11) and the fact that $q^{l}=p^{l}$ for every $l \in \bigsqcup_{m=k+1}^nA_m$.\\
Now we can finally complete the proof by contradiction:
\[||q-x||\leq||q-p||+||p-x||\leq1-\lambda+||p-x||<|x^i|-|y^i|+||p-x|| \leq||p-y||=||I_F(q)-I_F(x)||.\]
   The strict inequality in the chain above contradicts the fact that $I_F$ is non-expansive.
\end{proof}
The following theorem is a direct corollary of Lemma \ref{lem_main} and the fact that $M$ is dense in $S_X$.
\begin{theorem}
    Let $F:S_X \rightarrow S_X$ be a non-expansive bijection.
    Then $F$ is an isometry.
\end{theorem}

\end{document}